\documentclass[12pt]{article}
\usepackage{amsgen, mathtools, amsmath, amsfonts, amsthm, amssymb, enumerate,amscd,hyperref,graphicx}
\usepackage{tikz-cd}

\newtheorem{defn}{Definition}[section]
\newtheorem{prop}[defn]{Proposition}

\definecolor{blue}{rgb}{0,0,1}
\definecolor{red}{rgb}{1,0,0}
\definecolor{green}{rgb}{0,.6,.2}
\definecolor{purple}{rgb}{1,0,1}

\long\def\red#1\endred{\textcolor{red}{#1}}
\long\def\blue#1\endblue{\textcolor{blue}{#1}}
\long\def\purple#1\endpurple{\textcolor{purple}{ #1}}
\long\def\green#1\endgreen{\textcolor{green}{#1}}

\newtheorem{lem}[defn]{Lemma}
\newtheorem{thm}[defn]{Theorem}

\newtheorem*{conj*}{Conjecture}
\newtheorem{cor}[defn]{Corollary}

\newcommand {\ZZ}{{\mathbb Z}}

\newcommand {\C}{{\mathbb C}}

\newcommand {\G}{{\Gamma}}

\newcommand {\Q}{{\mathbb Q}}
\newcommand {\R}{{\mathbb R}}

\newcommand {\g}{{\gamma}}

\newcommand {\HH}{{\mathfrak  H}}

\newcommand {\MH}{\mathcal{M}}
\newcommand {\M}{{\mathcal M}}

\newcommand {\HM}{{\mathcal H \MH^!}}
\newcommand {\MI}{{\mathcal{MI}^!_1}}

\def\mod{\operatorname{mod}}

\def\Im{\operatorname{Im}}

\title{}
\author{ \\
}

\begin{document}

\title{Period functions associated to real-analytic modular forms}
\author
{Nikolaos Diamantis, Joshua Drewitt 
\\
(University of Nottingham)
}

\maketitle

\section{Introduction}
Period polynomials are fundamental objects associated to cusp forms which characterise the ``critical values" of their L-functions. They have been studied from various standpoints since at least the 70s and have been used to prove many important results including Manin's Periods Theorem. We state it in a slightly weakened form to make the comparison with one of our theorems easier: 
\begin{prop} \label{Manin} \cite{M} Let $f$ be a cusp form of weight $k$ for $\G:=$SL$_2(\ZZ)$ which is a normalised eigenfunction of the Hecke operators. Let $L^{*}_f$ denote the ``completed" L-function of $f$. Then there are $\omega_1(f), \omega_2(f) \in \C$ such that
$$L^*_f(j) \in \omega_1(f)K_f+ \omega_2(f)K_f \qquad \text{for each $j \in \{2, k-2\}$}$$
where $K_f$ is the field generated by the Fourier coefficients of $f$. 
\end{prop}

The analogue of period polynomials for Maass cusp forms proved to be harder to construct. It was introduced and studied by Lewis and Zagier in the late 90s (\cite{L, LZ}). This \emph{period function} found important uses in a variety of contexts, though not in arithmetic applications. 

On the other hand, F. Brown recently \cite{BrI, BrII, BrIII} initiated the study of a new class of automorphic objects, called \emph{real-analytic modular forms of weight $(r, s)$} whose behaviour is, in a sense to become clearer later, a hybrid of the behaviour of holomorphic forms and that of Maass forms. We denote their space by $\MH^!_{r, s}.$ He proved interesting algebraicity and rationality results for Fourier coefficients of elements of $\MH^!_{r, s}$ which affirmed their arithmetic nature. The space of real-analytic modular forms contains several previously studied classes of modular objects, including that of \emph{weakly holomorphic forms}.
\medskip

In this paper we investigate other fundamental arithmetic aspects of real-analytic modular forms, including their L-functions. 
The precise definition is stated in Sect. \ref{Lfun}  but in the special case of ``cuspidal" $f \in \MH^!_{r, s}$,  it is given by
\begin{multline*}\label{Lf}L_f^*(w):=\int_1^{\infty}\tilde f(it) t^{w-1}dt + \int_1^{-\infty}\mathring{f}(it) t^{w-1}dt
+ \\
i^{r-s} \left ( \int_1^{\infty} \tilde f(it)t^{r+s-w-1}dt +\int_1^{-\infty} \mathring{f}(it) t^{r+s-w-1}dt
\right ) 
\end{multline*}
where $\tilde f$ and $\mathring{f}$ stand for the ``pieces" of $f$ that are, respectively, exponentially decreasing and exponentially increasing at infinity.

\medskip

It is not obvious how to define appropriately period functions in $\mathcal M^!_{r, s}$.
This is not surprising because period functions are normally supposed to reflect arithmeticity and the full space $\mathcal M^!_{r, s}$ is too large to be of arithmetic nature in its entirety.  However, here we associate period functions to elements of a special subclass of $\mathcal M^!_{r, s}$, namely the subclass of  \emph{modular iterated integrals (of length $1$)} of \cite{BrI}. 
%To this end we use M\"uhlenbruch's construction of period functions for Maass forms of general weight \cite{MTh}. 
The definition (Sect. \ref{PerMaass}) requires some preparation, so we will illustrate the construction here under a simplifying assumption that does not hold in general.

\medskip

Fix $s, r \in \mathbb N$ of the same parity. For an integer $k  \le s, r$, $\zeta \in \C$ and a smooth $g: \HH \cup \bar \HH \to \C$, where $\HH$ (resp. $\bar \HH$) denote the upper (resp. lower) half-plane, we define the $1$-differential form
\begin{equation*}\label{omega}
\omega_k(g;\zeta)=
y^{k-1} \left ( \frac{\partial g}{\partial z}+rg(z) \right ) (\zeta-\bar z)^{s-k}(\zeta-z)^{r-k}dz+(s-k)y^{k-1} g(z)(\zeta-\bar z)^{s-k-1}(\zeta-z)^{r-k+1}d\bar z
\end{equation*}
where $z=x+iy$.

Let now $F \in \MH^!_{r, s}$ which is a modular integral (of length $1$). As shown in \cite{BrIII}, it can be uniquely decomposed as
$$F=F_0+\dots+F_{\text{\rm min}(r, s)}$$
where $F_k \in \MH^!_{r, s}$ is an eigenfunction of the Laplacian with eigenvalue $(k-1)(r+s-k).$ Our simplifying assumption, for the purposes of the Introduction, is that each $F_k$ has a vanishing ``constant term" (see Sect. \ref{basics}). We then set
$$v_k(\zeta)=\int_{\zeta}^{i \infty}\omega_k(\tilde F_k; \zeta )
+
\int_{\zeta}^{-i \infty}\omega_k(\mathring{F}_k; \zeta ).$$
We can now state
\begin{defn}Let $F \in \MH^!_{r, s}$ which is a modular integral. With the above notation, the \emph{period function $P(\zeta)$} of $F$ is given by
$$P(\zeta)=\sum_{k=0}^{\text{\rm min}(r, s)}\left ( v_k(-1/\zeta)\zeta^{r+s-2k}-v_k(\zeta) \right ).$$
\end{defn}
The period function induces a cocycle which is consistent with the Eichler cocycle of standard modular forms. Specifically, for each fixed $k \in \{0, \dots, \text{\rm min}(r, s)\}$, let $\G:= \text{SL}_2(\ZZ) $ act on the space $P_{r+s-2k}(\C)$ of polynomials of degree $\le r+s-2k$, via:
$$(P \underset{2k-r-s, 0}{||} \gamma)(\zeta):=P(\gamma \zeta) (c\zeta+d)^{r+s-2k} \qquad \text{for all $P \in P_{r+s-2k}(\C), \left ( \begin{smallmatrix} * & * \\ c & d \end{smallmatrix} \right ) \in \text{SL}_2(\ZZ) $}.$$
(The reason for the unusual notation $||$ will become clearer later). With this action, the \emph{Eichler cohomology group} is the group $H^1(\text{SL}_2(\ZZ),  P_{r+s-2k}(\C))$. 

For each fixed $k \in \{0, \dots, \text{\rm min}(r, s)\}$, let the map $\sigma_k:\text{SL}_2(\ZZ)  \to P_{r+s-2k}(\C)$ be given by
$$\sigma_k(\gamma)=(v_k \underset{2k-r-s}{||}\gamma)-v_k.$$
We have $P(\zeta)=\sigma_1(S)(\zeta)+\dots+\sigma_{\text{\rm min}(r, s)}(S)(\zeta).$ Further, define
$\tilde \sigma_k:\text{SL}_2(\ZZ)  \to P_{r+s-2k}(\C)$ as
$$\tilde \sigma_k(\gamma)=\int_i^{\gamma^{-1} i} \omega_k(F_k; \zeta ).$$
By comparison, the Eichler cocycle of a weight $k$ cusp form $f$ for $\G$ is $\sigma(\gamma)=\int_i^{\gamma^{-1} i} \omega(f; \zeta)$, with $\omega(f; \zeta)=f(z)(z-\zeta)^{k-2}dz.$
The cocycle $\sigma_k$ induced by our $P(\zeta)$ is consistent with the Eichler cocycle $\tilde \sigma_k$ in the following sense:
%, is the content of the following proposition.
\begin{prop} The maps $\sigma_k, \tilde \sigma_k$ define $1$-cocycles in $P_{r+s-2k}(\C).$ Their classes in the Eichler cohomology $H^1(\text{SL}_2(\ZZ),  P_{r+s-2k}(\C))$ coincide.
\end{prop}

The connection with our $L$-function is provided by the following theorem (see Th. \ref{Lperiod}).
\begin{thm} \label{Lv} Assume that $r \equiv s \mod 4$. Let $P'(\zeta)$ denote the polynomial obtained from $P(\zeta)$ upon removing its constant and leading terms. Then,
$$P'(\zeta)=\sum_{k=0}^{\text{\rm min}(r, s)}\sum_{l=1}^{r+s-2k-1}a_{k, l} L_{F_k}^*(k+l)\zeta^l$$
for some explicit $a_{k, l} \in \Q(i).$
\end{thm}

Finally, as an additional evidence that our definition of L-function is the``right" one, we prove the following:
\begin{thm} \label{ManinW} The analogue of Proposition \ref{Manin} for $f$ \emph{weakly holomorphic} holds.
\end{thm}
This differs from the full Manin's Periods Theorem in that it does not say anything about the values at $1$ and $k-1$. 
%The fact that the lack of $K_f$-proportionality of $L^*_f(1)$ (and of  $L^*_f(k-1)$) with the other odd ``critical values" seems to be confirmed by numerical experiments.
As confirmed by numerical experiments, the lack of 
$K_f$-proportionality of $L^*_f(1)$ (and of  $L^*_f(k-1)$) with the other odd ``critical values" seems to be genuine and not just due to any incompleteness of our proof. 
  
The precise statement of Th. \ref{ManinW} is Th. \ref{ManinWp}. To prove it, we use an analogous identity to Th. \ref{Lv} and the algebraic de Rham theory of weakly holomorphic modular forms \cite{BH}. K. Bringmann has shown %us how to apply an identity of \cite{BGKO} which, combined with Th. \ref{Lv}, implies Th. \ref{ManinW}.
us how we can use an identity of \cite{BGKO} to deduce a statement which, combined with Th. \ref{Lv}, implies Th. \ref{ManinW}.

\bigskip

{\bf Acknowledgements.} We thank F. Brown for many helpful comments on the exposition and possible further directions of research, K. Bringmann for suggesting an alternative approach to Th. \ref{ManinW} and F. Str\"omberg for numerical tests of some of the results. 

\section{Real-analytic modular forms}\label{basics}
We start by recalling the definition of real-analytic modular forms.

Let $\G=$SL$_2(\ZZ)$ and set $S=\left ( \begin{smallmatrix} 0 & -1 \\ 1 & 0 \end{smallmatrix} \right )$, 
$T=\left ( \begin{smallmatrix} 1 & 1 \\ 0 & 1 \end{smallmatrix} \right )$ and 
$U=TS=\left ( \begin{smallmatrix} 1 & -1 \\ 1 & 0 \end{smallmatrix} \right ).$  Suppose that $r, s$ are integers of the same parity. For a smooth $g: \HH \to \C$ and $\g \in \G$ we define a function $f\underset{r, s}{||}  \g$ given by 
$$f \underset{r, s}{||} \g(z)=j(\g, z)^{-r} j(\g, \bar z)^{-s }f(\g z) \qquad \text{for all $z \in \HH$ }$$
where
$$j\left ( \left ( \begin{smallmatrix} * & * \\ c & d \end{smallmatrix} \right ), z \right ):=cz+d.$$
We extend the action to $\C[\G]$ by linearity.

 We call a real-analytic function $f: \HH \to \C$ \emph{a real-analytic modular form of weights $(r, s)$ for $\G$} if \newline
1. for all $\g \in \G$ and $z \in \HH$, we have $f \underset{r, s}{||} \g=f$, i.e.
$$f(\g z)=j(\g, z)^r j(\g, \bar z)^s f(z)  \qquad \text{for all $z \in \HH$. }$$
\noindent
2. for some $C>0$, $f(z)=O(e^{Cy})$ as $y \to \infty$, uniformly in $x$ (where $z=x+iy$). Further,
\begin{equation} \label{FE} f(z)=\sum_{|j| \le M}y^j\left ( \sum_{m, n \ge -N} a_{m, n}^{(j)} q^m \bar q^n \right ) \end{equation}
 for some $M, N \in \mathbb N,$ $a_{m, n}^{(j)} \in \C$. Here, $q=$exp$(2 \pi i z)$. 

We denote the space of real analytic modular forms of weights $(r, s)$ for $\G$ by $\MH^!_{r, s}.$
We set $\MH^!=\oplus_{r, s }\MH^!_{r, s}.$
This class of functions was introduced by F. Brown \cite{BrI, BrII, BrIII} whose initial motivation was related, on the one hand, to some non-holomorphic modular forms originating from iterated extensions of pure motives and with coefficients that are periods. On the other hand, he was motivated by open questions about modular graph functions appering in string perturbation theory.

The space $\MH^!$ contains or intersects various previously studied classes of important modular objects and the point of view we adopt here is to consider real-analytic forms as a unifying tool for those classes. For example, 
when $s=0$, an element $f$ of $\MH^!_{r, 0}$ which is holomorphic in $\HH$ is a standard weakly holomorphic modular form
of weight $r$ for $\G.$ We denote their space by $M_r^!$ and set $M^!=\oplus_r M^!_r.$ The space $M_r^!$ contains, of course, the space $M_r$ (resp. $S_r$) of classical modular (resp. cusp) forms of weight $r$.

When $r=0$ we are similarly led to the space $\bar M_r$ of weakly anti-holomorphic modular forms.

Another subspace is $\MH_{r, s}$, which is obtained upon imposing the condition that $a_{m, n}^{(k)}$ should vanish when $m$ or $n$ are negative. It was defined and studied in \cite{BrI}. Their direct sum over all $r, s$ is denoted by $\MH.$

The relation with Maass forms is more complicated. On the one hand, the definition of $\MH^!$ allows for forms which are not eigenfunctions of the Laplacian, but, on the other, it requires a more restrictive form of a Fourier expansion than that of Maass forms. We will exploit this relation in the sequel in order to define some of our main objects, and, in particular, we will see that constructions from the theory of Maass cusp forms will be the basis for period functions of certain elements of $\MH^!$.

The Fourier expansion \eqref{FE} can be uniquely decomposed into a sum of an ``principal part" $\mathring{f}$, an exponentially decaying part $\tilde f$ and the ``constant term" $f^0$ defined as follows:
\begin{align} \label{tilde} \tilde f (z) &:= \sum_{|j| \le M}y^j\left ( \sum_{\substack{m, n \ge -N \\ m+n>0}} a_{m, n}^{(j)} q^m \bar q^n \right ) 
\\
\label{0} f^0(z)&:= \sum_{|j| \le M}y^j\left ( \sum_{\substack{m, n \ge -N \\ m+n=0}} a_{m, n}^{(j)} q^m \bar q^n \right ) 
\qquad \text{and} \\
\mathring{f}(z)&:= f(z)-\tilde{f}(z)-f^0(z)
 \label{ring} 
\end{align} 
 
\subsection{Eigenforms for the Laplacian}
The Lie algebra $\mathfrak{sl}_2$ acts on $\MH^!$ via the Maass operators $\partial_{r}: \MH^!_{r, s} \to \MH^!_{r+1, s-1}$
and $\bar \partial_s: \MH^!_{r, s} \to \MH^!_{r-1, s+1}$ given by
$$\partial_r=2iy\frac{\partial}{\partial z}+r \qquad \text{and} \, \, \,  \bar \partial_{s}=-2iy\frac{\partial}{\partial \bar z}+s.$$
They induce bigraded derivations on $\MH^!$ denoted by $\partial$ and $\bar \partial$ respectively.
 
The Laplacian $\Delta_{r, s}: \MH^!_{r, s} \to \MH^!_{r, s}$ is defined by
$$\Delta_{r, s}=-\bar \partial_{ s-1} \partial_{r} +r(s-1)=-\partial_{r-1}\bar \partial_{s}+s(r-1).$$
It induces a bigraded operator  $\Delta$ of bidegree $(0, 0)$ on $\M^!$.

An operator which is essentially equivalent to $\Delta_{r, s}$ but which is more convenient for some computations in the sequel is
$$\Omega_k=-y^2\left ( \frac{\partial^2}{\partial x^2}+\frac{\partial^2}{\partial y^2} \right )+i k y \frac{\partial}{\partial x}.$$
When working with $\Omega_k$ the following version of the `stroke' operator will be more appropriate to work with than $\underset{r, s}{||}$.
Specifically, for $\gamma \in \G$ and $f: \mathfrak H \to \mathbb C$, the function $f|_k\g$ is defined by 
$$(f|_k\g)(z)=\left ( \frac{j(\g, z)}{|j(\g, z)|} \right ) ^{-k}f(\g z) \qquad {\text{for all $z \in \HH.$}}$$
We extend the action to $\C[\G]$ by linearity.
To move between the $\Delta$ to the $\Omega$ formalism 
 the following lemma will be useful:
\begin{lem}\label{trans} If,  for some $r, s \in \ZZ$ of the same parity, $F$ is an element of $\MH^!_{r, s}$ such that $\Delta_{r, s}F=\lambda F$ for some $\lambda \in \R$, then
$F_1:=y^{\frac{r+s}{2}}F$ satisfies
$$F_1|_{r-s} \g=F_1 \qquad \text{for all $\g \in \G$ and} \, \, \, \Omega_{r-s}F_1=\left ( \lambda+\left( \frac{r+s}{2} \right ) \left (1-\frac{r+s}{2}\right )\right ) F_1.$$ 
\end{lem}
\begin{proof} We observe that 
$$\Delta_{r, s}=\Omega_{r-s}-(r+s)y \frac{\partial}{\partial y}$$
and that, for each smooth $f: \HH \to \C$,
$$f|_k\g=f \underset{\frac{k}{2}, -\frac{k}{2}}{||}\g \qquad \text{and} \, \, \left ( y^{\frac{r+s}{2}} f \right ) |_{r-s} \g= y^{\frac{r+s}{2}} \left (f  \underset{r, s}{||}\g \right ) \quad  \text{for all $\g \in \G$.}$$
An easy computation then implies the lemma. 
\end{proof}

For $\lambda \in \R,$ set
$$\HM(\lambda):=\text{Ker}(\Delta-\lambda:\M^! \to \M^!).$$
Also set $\HM:=\oplus_{\lambda} \HM(\lambda).$
The following lemma summarises some of the special features of the Fourier expansions of $f \in \HM(\lambda)$.
\begin{lem} \label{HM!}\cite{BrIII} Let $f \in \HM(\lambda) \cap \M^!_{r, s}$. Then,  $\lambda \in \ZZ$ and there is a $k_0 \in \ZZ$ such that $k_0<1-r-s-k_0$ and $\lambda=k_0(1-r-s-k_0).$  There are unique $f^h, f^a$ of the form 
%\begin{align} f^0(z)&= ay^{k_0}+by^{1-r-s-k_0} \\
\begin{align} f^h(z)&=\sum_{j=k_0}^{-s} y^j\sum_{\substack{m \ge -N \\ m \ne 0}}a_{m}^{(j)} q^m \label{fh}\\
f^a(z)&=\sum_{j=k_0}^{-r} y^j\sum_{\substack{m \ge -N' \\ m \ne 0}} b_{m}^{(j)} \bar q^m \label{fa}
\end{align} 
($a_{m}^{(j)},  b_{m}^{(j)} \in \C $ and $N, N' \in \mathbb N$) such that
$$f=f^0+f^h+f^a.$$
Furthermore, the constant term has the form $ f^0(z)= ay^{k_0}+by^{1-r-s-k_0}$ for some $a, b \in \mathbb C$. 

Finally, $f^h, f^a, y^{k_0}$ and $y^{1-r-s-k_0}$ are eigenfunctions of $\Delta_{r, s}$ with eigenvalue $\lambda$.
\end{lem}
\begin{proof} Lemma 4.3 of \cite{BrIII} together with the remarks following it. 
\end{proof}

\subsection{Real analytic Eisenstein series.} \label{reES} An example of an element of $\MH^!_{r, s}$, and, indeed, of $\MH_{r, s}$, which is, in addition, an eigenform for the Laplacian is the \emph{real analytic Eisenstein series} $\mathcal E_{r, s}$, given for $r, s \in \mathbb N$ and $z \in \mathfrak H$ by $$\mathcal E_{r, s}(z)=\frac{1}{2}\sum_{\g \in B \backslash
\G} \frac{\Im(z)}{j(\gamma, z)^{r+1} j(\gamma, \bar z)^{s+1}}
$$
where $B$ is the subgroup of translations. This series converges absolutely and belongs to $\MH_{r,s}$
It further has a meromorphic continuation to the entire complex plane and is an eigenfuction of $\Delta$ with eigenvalue $-r-s.$

Its Fourier expansion has been computed explicitly in Th. 3.1. of \cite{DO1} and in Prop. 11.2.16 of \cite{CS}. We summarise it here and will see how it fits with Lemma \ref{HM!}. With the notation of that lemma,
\begin{multline}\label{FEis} \mathcal E_{r, s}(z)=\mathcal E^0_{r, s}(z)+\frac{ \pi^{\frac{r+s+2}{2}}}{\Gamma(1+r)\zeta(r+s+1)} \left \{ \sum_{j=-r-s}^{-s}y^j \sum_{m \ge 1} \frac{\sigma_{r+s+1}(m)\alpha^+_{-j-\frac{r+s}{2}}(4 \pi)^{j+\frac{r+s}{2}}}{m^{1-j}}q^m \right. \\+
\left. \sum_{j=-r-s}^{-r}y^j \sum_{m \ge 1} \frac{\sigma_{r+s+1}(m)\alpha^-_{-j-\frac{r+s}{2}}(4 \pi)^{j+\frac{r+s}{2}}}{m^{1-j}}\bar q^m  \right \}
\end{multline}
where
$$\alpha_j^{\pm}:=(-1)^j(j+\frac{|r-s|}{2})! \binom{s+\frac{r-s}{2}+\frac{|r-s|}{2}}{j+\frac{|r-s|}{2}} 
\binom{(\pm1-1) \frac{r-s}{2}-1-s}{j \pm \frac{r-s}{2}}.$$
Here $\binom{a}{b}$ with $a<0$ are defined in accordance with the 
convention that, if $a<0$ and $j \ge 0$, then $\binom{a+j}{j}=(a+j)(a+j-1) \dots (a+1)/j!.$
Thus
$k_0=-r-s$ and
$$a_m^{(j)}=\frac{ (2 \pi)^{r+s+j} 2^j \pi \alpha^+_{-j-\frac{r+s}{2}}}{\Gamma(1+r)\zeta(r+s+1)}  \sigma_{r+s+1}(m)m^{j-1} \quad \text{and}
 \, \, b_m^{(j)}=\frac{ (2 \pi)^{r+s+j} 2^j \pi \alpha^-_{-j-\frac{r+s}{2}}}{\Gamma(1+r)\zeta(r+s+1)}  \sigma_{r+s+1}(m)m^{j-1}$$
which is consistent with Lemma \ref{HM!}.

\subsection{Modular iterated integrals of length one}
In \cite{BrIII}, the space of modular iterated integrals is defined. We will consider only the special case of length one: The space $\MI$ of modular iterated integrals of length one is defined to be the largest subspace of $\mathcal M^!$ which satisfies
\begin{align} \partial \MI & \subset \MI+M^![y^{\pm}] \nonumber \\
\bar \partial \MI & \subset \MI+\overline{ M^!}[y^{\pm}]
\end{align}
A characterisation of this space is provided in \cite{BrIII}:
\begin{prop}\label{char} (Prop. 5.8 of \cite{BrIII}) Any element $F$ of $\MI$ of weights $r, s$ can be written uniquely as
$$F=\sum_{k=0}^{\text{\rm min}(r, s)}F_k$$
for some elements of $\MI$ of weights $r, s$ such that
$\Delta_{r, s}F_k=\lambda_k F_k$ 
where $$\lambda_k=(k-1)(r+s-k).$$
This, in particular, implies that the value of the invariant $k_0$ (see Lemma \ref{HM!}) for $F_k$ is
$$k_0=k-r-s.$$
\end{prop}

We will interpret the functions $F_k$ of the last proposition in the setting of the last section.
\begin{prop} \label{Fk} Let $F$ be an element of $\MI$ with weights $r, s$ and let $F_k$ ($k \in \{0, \dots, \text{\rm min}(r, s)\}$) 
be as in Prop. \ref{char}. Then $y^{\frac{r+s}{2}}F_k(z)$ is $\G$-invariant under the action of $|_{r-s}$ and is an eigenfunction of $\Omega_{r-s}$ with eigenvalue 
$$\frac14-\mu_k^2 \qquad \qquad \text{where $\mu_k=- k+\frac{r+s+1}{2}.$}$$ 
%Further,  $F_k(z)$ is uniquely decomposed as a sum of an 
%The ``principal part" $\mathring{F_k}$ and the exponentially decaying part $\tilde F_k$ 
%\begin{align*} \mathring{F_k} (z)&=F_k^0(z)+\sum_{j=k_0}^{-s} y^j\sum_{-N \le m <0} a_{m}^{(j)} q^m 
%+\sum_{j=k_0}^{-r} y^j\sum_{ -N'\le m <0} b_{m}^{(j)} \bar q^m \\
%tilde{F_k}(z)& =F_k(z)-\mathring{F_k}(z).
%\end{align*} Finally, 
%satisfy 
For $g=\tilde F_k, \mathring{F}_k, y^{k-r-s}, y^{1-k}$ (in the notation \eqref{tilde}, \eqref{ring}) we have
\iffalse
\begin{equation}\label{eprop}\Omega_{r-s} \left ( y^{\frac{r+s}{2}} \mathring{F_k} \right ) =\left(\frac14-\mu_k^2 \right )y^{\frac{r+s}{2}}\mathring{F_k} \qquad \text{and} \, \, \, 
\Omega_{r-s}\left ( y^{\frac{r+s}{2}}\tilde{F_k} \right ) =\left(\frac14-\mu_k^2 \right )y^{\frac{r+s}{2}} \tilde{F_k}.
\end{equation}
\fi
\begin{equation}\label{eprop}\Omega_{r-s} \left ( y^{\frac{r+s}{2}} g \right ) =\left(\frac14-\mu_k^2 \right )y^{\frac{r+s}{2}}g 
\end{equation}
\end{prop}
\begin{proof}
By Lemma \ref{trans} and Prop. \ref{char}, $y^{\frac{r+s}{2}}F_k(z)$ is $\G$-invariant under the action of $|_{r-s}$ and eigenfunction of $\Omega_{r-s}$ with eigenvalue
$$\lambda_k+\left ( \frac{r+s}{2} \right ) \left ( 1-\frac{r+s}{2} \right )=
(k-1)(r+s-k)+\left ( \frac{r+s}{2} \right ) \left ( 1-\frac{r+s}{2} \right )=\frac14-\left ( k-\frac{r+s+1}{2}\right )^2.$$

%The decomposition in terms of $\mathring{F}_k, \tilde{F}_k$ follows immediately from the decomposition of Lemma \eqref{HM!}.

To establish the eigen-properties \eqref{eprop}  we apply the last statement of Lemma \ref{HM!} to $F_k \in \HM(\lambda_k).$ Then since $\Delta_{r, s}F_k^h=\lambda_k F_k^h$ and $\Delta_{r, s}F^0_k=\lambda_kF^0_k$, we have
\begin{equation}\label{e.f.} \left ( \Delta_{r, s}-\lambda_k \right ) \left ( 
\sum_{j=k_0}^{-s} y^j\sum_{-N \le m <0} a_{m}^{(j)} q^m \right )=-
\left ( \Delta_{r, s}-\lambda_k \right ) \left ( 
\sum_{j=k_0}^{-s} y^j\sum_{m >0} a_{m}^{(j)} q^m \right )
\end{equation}
Let $\mathcal P$ be the space of polynomials of $y$ over $\C$. By (2.22) of \cite{BrI} we have
\begin{equation}\label{Deltaaction}
\Delta_{r, s}(\mathcal P \cdot q^m \bar q^{n}) \subset \mathcal P \cdot q^m \bar q^{n}
\end{equation}
for each $m, n \in \ZZ$. Therefore, the LHS of \eqref{e.f.} will be a polynomial in $q^{-1}$ with coefficients in $\mathcal P$ and thus, if not identically $0$, it will have exponential growth as $y \to \infty$. This is impossible because, by \eqref{Deltaaction}, the RHS of \eqref{e.f.} decays exponentially as $y \to \infty$. Therefore the LHS vanishes and 
$$\Delta_{r, s} \left ( \sum_{j=k_0}^{-s} y^j\sum_{-N \le m <0} a_{m}^{(j)} q^m \right )=\lambda_k \left ( \sum_{j=k_0}^{-s} y^j\sum_{-N \le m <0} a_{m}^{(j)} q^m \right ).$$

We similarly see that $\sum_{j=k_0}^{-r} y^j\sum_{-N' \le m <0} b_{m}^{(j)} \bar q^m$ is an eigenfunction of $\Delta_{r, s}$ with eigenvalue $\lambda_k$. Thus $\mathring{F}_k$ is an eigenfunction of $\Delta_{r, s}$ with eigenvalue $\lambda_k.$

Since $y^{k-r-s}, y^{1-k}$ are also eigenfunctions of $\Delta_{r, s}$ with eigenvalue $\lambda_k$, we deduce %that the sum of the three terms we considered is an eigenfunction of $\Delta_{r, s}$ with eigenvalue $\lambda_k$. 
(with Lemma \ref{trans}) the desired eigenproperties of $y^{(s+r)/2}y^{k-r-s}$ and $y^{(s+r)/2}y^{1-k}$. The eigenproperties of  $y^{(s+r)/2}\mathring{F}_k$,
$y^{(s+r)/2}y^{k-r-s}, y^{(s+r)/2}y^{1-k}, $ just proved, together with the eigenproperty of  $y^{(s+r)/2}F_k$ then imply the eigenproperty
of $y^{(s+r)/2}\tilde F_k$.
\end{proof}

\section{L-functions}\label{Lfunctions}
The obstacles to extending the definition of L-functions of standard modular forms to $\HM$ are due to the potentially exponential growth of functions in $\mathcal M^!$ combined with the lack of holomorphicity. To tackle the former we can give a definition that is based on the expression of standard L-functions through Mellin transforms. This will, in fact, allow us to define L-functions on the entire $\mathcal M^!.$ 

\subsection{L-functions in $\mathcal M^!$.}\label{Lfun}
Let $f \in \mathcal M^!_{r, s}$ with an expansion \eqref{FE}. 
We let the implied logarithm take the principal branch of the logarithm and we set, for $w \ne -j, r+s+j$ ($|j| \le M$),
\begin{multline}\label{Lf}L_f^*(w):=\left ( \int_1^{\infty}\tilde f(it) t^{w-1}dt + \int_1^{-\infty}\mathring{f}(it) t^{w-1}dt
-\sum_{|j| \le M} \sum_{\substack{m, n \ge -N \\ m+n=0}}\frac{ a_{m, n}^{(j)}}{w+j} \right )+ \\
i^{r-s} \left ( \int_1^{\infty} \tilde f(it)t^{r+s-w-1}dt +\int_1^{-\infty} \mathring{f}(it) t^{r+s-w-1}dt
-\sum_{|j| \le M} \sum_{\substack{m, n \ge -N \\ m+n=0}}\frac{ a_{m, n}^{(j)}}{r+s-w+j}
\right ) .
\end{multline}
The rigorous meaning of the first integral from $1$ to $-\infty$ is 
\begin{equation}\label{rigo} 
\sum_{|j| \le M} \sum_{\substack{m, n \ge -N \\ m+n<0}} \frac{a_{m, n}^{(j)}}{(2 \pi (m+n))^{j+w}} \Gamma(j+w, 2 \pi (m+n)) 
\end{equation}
where $\Gamma(r, z)$ denotes the incomplete Gamma function
\[
  \Gamma(r,z) := \int_{z}^\infty e^{-t}t^{r}\, \frac{dt}{t}.
\]
For $z \ne 0$, this has an analytic continuation to the entire $r$-plane and therefore, \eqref{rigo}
is well-defined for all values of $w$ by the analytic continuation of incomplete Gamma function. By contrast, the real integral as written in \eqref{Lf} is not convergent at $0$ unless $\Re(w)>1+M.$ 
We interpret likewise the second integral from $1$ to $-\infty$ in \eqref{Lf}. The reason we preferred to write formally those terms as integrals was to stress the symmetry with the other terms and hint at the origin of the definition in a `regularisation' introduced in \cite{Br}.   
%(Here and in the sequel we use the principal branch of the logarithm).
   
Since, in addition, $\tilde f$ decays exponentially at infinity, all integrals in \eqref{Lf} are well-defined.  As mentioned above, the above construction was inspired by the `regularisation' introduced in \cite{Br}, Sect. 4. (See \cite{DR}, for another application of this idea.) 
 
The definition immediately implies the following:
\begin{prop} Let $f \in \mathcal M^!_{r, s}$ (with $r \equiv s \mod 2$). The L-function of $f$ is meromorphic with finitely many poles and satisfies
$$L_f^*(w)=i^{r-s}L_f^*(r+s-w)$$
for all $w$ away from the poles. 
\end{prop}

 In such generality, the definition is somewhat formal and would be unlikely to lead to arithmetic insight for all $f \in \mathcal M^!_{r, s}$. To obtain more refined information, we restrict to subspaces of $ \mathcal M^!_{r, s}.$ 

We first note that, in the subspace $\mathcal M_{r, s}$ of $f \in \mathcal M^!_{r, s}$ with moderate growth at infinity, our definition coincides with that of Sect. 9.4 of \cite{BrI}. Specifically, in that case, $\mathring f=0$ and the Fourier coefficients of $f$ have polynomial growth. For Re$(w) \gg 0,$ the change of variable $t \to 1/t$ in the third integral of \eqref{Lf} together with the transformation law of $f$ implies that $L_f^*(w)$ coincides with the function $\Lambda(f, w)$ of  Sect. 9.4. of \cite{BrI}. See also, Section 9.4 of \cite{BrII} where this construction is applied to the important subclass of  $\mathcal M_{r, s}$ consisting of modular analogues of the single-valued polylogarithms.

\subsection{L-functions in $\HM(\lambda)$ and in $\MI.$}
Let $f \in \HM(\lambda) \cap \M^!_{r, s}$. Using the Fourier expansion of $f$ provided by Lemma \ref{HM!}, the general definition of $L^*_f(w)$ we gave above leads to an expression as a series.  This is more natural because it is reminiscent of the original definition of $L$-series of standard modular forms and because, in the case of weakly holomorphic modular forms, it coincides with the  L-functions already associated with such forms (\cite{BDE} and references therein).

To ensure that the series we will eventually obtain converges absolutely, we need an analogue of the ``trivial bound" about the Fourier coefficients. 
As in the case of weakly holomorphic forms (\cite{BF}, Lemma 3.2), the growth is, in general, exponential. 
Although the proof parallels that of \cite{BF}, there are some complications because of the presence of two weights and of the powers of $y$, so we present a full proof.
\begin{prop}\label{bound} Let $f \in \HM(\lambda)$. With the notation of Lemma \ref{HM!}, for each $j \in \{k_0, \dots, -s\}$ (resp. $j \in \{k_0, \dots, -r\}$), there is a $C>0$ such that, 
$$a_n^{(j)} \ll e^{C \sqrt{n}} \qquad (\text{resp. $b_n^{(j)} \ll e^{C \sqrt{n}}$}) \qquad \text{as $n \to \infty$}$$
\end{prop}
\begin{proof}  
%Assume, without loss of generality that $N \ge N'$, in the notation of Lemma \ref{HM!}. 
Set $N_0=$max$(N, N')$ and let $n>N_0.$ Then we have
\begin{multline}\label{a}
\int_0^1f(z)e^{-2 \pi i n z} dx=
\int_0^1f^0(z)e^{-2 \pi i n z} dx+\\
\sum_{j=k_0}^{-s} y^j\sum_{\substack{m \ge -N \\ m \ne 0}}a_{m}^{(j)} \int_0^1e^{2 \pi i(m-n)z}dx +
\sum_{j=k_0}^{-r} y^j\sum_{\substack{m \ge -N' \\ m \ne 0}} a_{m}^{(j)} e^{2 \pi (n-m)y}\int_0^1 e^{-2 \pi i (m+n)x}dx
\\=\sum_{j=k_0}^{-s} y^j a_{n}^{(j)}
\end{multline}
since, for $n>N_0$ and $m \ge -N',$ $m+n>0$. 
Likewise,
\begin{equation}\label{b} \int_0^1f(z)e^{2 \pi i n z} dx=\sum_{j=k_0}^{-r} y^j b_{n}^{(j)}e^{-4 \pi n y}.
\end{equation}
Suppose that $-s>k_0$. Then \eqref{a} implies
$$e^{2 \pi n y} y^{-k_0} \int_0^1f(z)e^{-2 \pi i n x} dx=\sum_{j=0}^{-s-k_0} y^j a_n^{(j+k_0)}$$
and thus
\begin{equation}\label{Fcoeff}
(-s-k_0)! a_n^{(-s)}= \frac{\partial^{-s-k_0}}{\partial y^{-s-k_0}} \left ( e^{2 \pi n y} y^{-k_0} \int_0^1f(z)e^{-2 \pi i n x} dx \right ).
\end{equation}
The RHS will be a sum of products of $e^{2 \pi n y}$, polynomials in $y$ and $n$  and 
\begin{equation}\label{int} \int_0^1\frac{\partial^j f(z)}{\partial y^j} e^{-2 \pi i n x} dx \end{equation}
for $j \in \{0, \dots, -s-k_0\}$. Now, we note that, for all $k, l$, 
\begin{equation} \label{Maass} \frac{\partial}{\partial y}=\frac{1}{2y}\left (\partial_k+\bar \partial_l-k-l \right ) \end{equation}
and, since $y^{-1} \in \M^!_{1, 1}$, we have  $\frac{\partial}{\partial y} \left ( \M^! \right ) 
%\subset  \M^!_{r, s+2} +\M^!_{r+2, s} +\M^!_{r+1, s+1} 
\subset \M^!$. Thus
%.Using \eqref{Maass} with $k=r, l=s+2$, we deduce that 
%$\frac{\partial}{\partial y} \left ( \M^!_{r+2, s} \right ) \subset  \M^!_{r+2, s+2} +\M^!_{r+4, s} +\M^!_{r+3, s+1} \subset \M^!$.
%and similarly for $\M^!_{r, s+2}$. Therefore
\begin{equation}\label{part} 
\frac{\partial^j }{\partial y^j} \left ( \M^! \right ) \subset \M^!.
\end{equation}
We will use this to bound \eqref{int} with the help of this lemma:
%\newpage
\begin{lem}\label{fbound} For each $f \in  \M^!_{r, s}$, there is a $C>0$ such that 
$f(z)=O(e^{C/y}y^{-(r+s)/2})$
as $y \to 0$, uniformly in $x$.
\end{lem}
\begin{proof}
%The function $f$ will have a decomposition as in Lemma \ref{HM!}. 
The standard $j$-function $j(z)$ does not vanish in the interior of the standard fundamental domain $\mathcal F$ of $\G$ Therefore, if $c>0$ is such that $f(z)=O(e^{cy})$ as $y \to \infty$ (uniformly in $x$), then $F(z)=|f(z)y^{\frac{r+s}{2}}j(z)^{-c-\epsilon}|$ is bounded in the standard fundamental domain $\mathcal F$ of $\G$.
%since the exponential growth of $f(z)$ is cancelled by the exponential decay of $j(z)^{-N_0}$. 
%Therefore,
%By the expansion \eqref{FE} of $f$, we that, as $y \to \infty$  within the standard fundamental domain $\mathcal F$ of $\G$,
%$$|f(z)y^{\frac{r+s}{2}}|=O(e^{C y})$$ for some $C>0$, independently of $x.$ 
On the other hand, 
$|f(z)y^{\frac{r+s}{2}}|$ is invariant under any $\g \in \G$ and thus under $S$. This implies that there is a $C>0$, such that, as $y \to 0$, 
$$|f(z)y^{\frac{r+s}{2}}|=|f(-1/z) \text{Im}(-1/z)^{\frac{r+s}{2}}|=O(e^{C \text{Im}(-1/z)})=O(e^{C/y})$$
uniformly in $x$, which gives the result.
\end{proof}
With this lemma and \eqref{part} we see that, for $y \to 0,$
$$\frac{\partial^j f}{\partial y^j} =O(e^{C/y}y^M)$$
for some constant depending on $r, s.$ From \eqref{Fcoeff}, we then conclude
$a_n^{(-s)} \ll e^{2 \pi ny} e^{C/y}y^{M_1}N^{M_2}$, where $C, M_1, M_2$ and the implied constant depend only on $r, s$ and $k_0$. 
For $y=1/\sqrt{n}$ this implies the bound of the proposition in the case $j=-s.$

Next, we differentiate in $y$ both sides of 
$$\sum_{j=0}^{-s-k_0-1} y^j a_n^{(j+k_0)}=-y^{-s-k_0}a_n^{(-s)}+e^{2 \pi n y} y^{-k_0} \int_0^1f(z)e^{-2 \pi i n x} dx$$
to get
$$ (-s-k_0-1)! a_n^{(-s-1)}=-(-s-k_0)! y a_n^{(-s)}+\frac{\partial^{-s-k_0-1}}{\partial y^{-s-k_0-1}} \left ( e^{2 \pi n y} y^{-k_0} \int_0^1f(z)e^{-2 \pi i n x} dx \right ).$$
Arguing as above for the second term, and using the bound for $ a_n^{(-s)}$ we proved above, we deduce the bound for $j=-s-1$. Continuing in this way, we deduce the result for all $j.$

It is clear from the argument (essentially by interchanging the roles of $s$ and $k_0$), that it remains valid when $-s \le k_0.$

To prove the bound for $b_n^{(j)}$, we work in the same way but based on \eqref{b}, instead of \eqref{a}.

\end{proof}
We are now ready to use the Fourier expansion given in Lemma \ref{HM!} to express the L-function of an $f \in \HM(\lambda)$ as a series.
For compactness of notation, we set, for each $j \in \mathbb Z,$ $c_m^{(j)}=a_m^{(j)}+b_m^{(j)}$, where $a_m^{(j)}$ (resp. $b_m^{(j)}$) are taken to be $0$ if $j$ or $m$ is outside the range of $j$- or $m$-summation in \eqref{fh} (resp. \eqref{fa}). Then, by substituting the Fourier expansion of $f$ into \eqref{Lf}, we deduce, for $w \ne -k_0, -k_0+1, k_0+r+s-1, k_0+r+s$,
%Suppose $r \le s.$ Set $c_m^{(j)}=a_m^{(j)}+b_m^{(j)}$, if $j \ge -r$ and $c_m^{(j)}=a_m^{(j)}$, otherwise. Similarly, if $s \le r.$ Then, we set 
\begin{multline}\label{Ldef}
L_f^*(w)=\sum_{j \in \mathbb Z} \sum_{m \ne 0} \frac{c^{(j)}_f(m)\Gamma(w+j, 2 \pi m )}{(2 \pi m)^{w+j}}+i^{r-s}\sum_{j \in \mathbb Z} \sum_{m \ne 0} \frac{c^{(j)}_f(m)\Gamma(s+r+j-w, 2\pi m )}{(2 \pi m)^{s+r+j-w}}+P(w)
\end{multline}
 where $P(w)$ denotes
$$\int_0^1(i^{r-s}f^{0}(\frac{i}{t})t^{-r-s}-f^0(it))t^{w}\frac{dt}{t}=\frac{i^{r-s}a}{w-k_0-r-s}+\frac{i^{r-s}b}{w+k_0-1}-\frac{a}{w+k_0}-\frac{b}{w-k_0-r-s+1} .$$
Because of Prop. \ref{bound} and the asymptotics $\Gamma(r, x) \sim e^{-x}x^{r-1}$ as $x \to \infty$ we see that this series  converges absolutely for all $w \in \mathbb C.$ 

\subsubsection{Example: L-function of a weakly holomorphic modular form} In the special case of a weakly holomorphic form, this formula coincides with the earlier definition of an L-function for such forms. Indeed, an $f \in M^!_k$ can be considered as an element of $\HM(0) \cap \M^!_{k, 0}$ with $k_0=1-k, a=b_m^{(j)}=0$ for all $j, m$ and $a_m^{(j)}=0$ for $j \ne 0$ or $m<m_0$ for some $m_0 \in \mathbb Z.$
Then \eqref{Ldef} becomes
$$L_f^*(w)=\sum_{m \ge m_0} \frac{a^{(0)}_f(m)\Gamma(w, 2 \pi m )}{(2 \pi m)^{w}}+i^{k} \sum_{m \ge m_0} \frac{a^{(0)}_f(m)\Gamma(k-w, 2 \pi m)}{(2 \pi m)^{k-w}}-b \left ( \frac{1}{w}+\frac{i^k}{k-w}\right )$$
which coincides with, say, (6.1) of \cite{BDE}.

\subsubsection{L-functions of modular iterated integrals of length $1$.}
Using Prop. \ref{char}, we can now express the L-function of a function $F$ in the broader class $\MI$ of modular iterated integrals of length $1$, in terms of the L-function in $\HM(\lambda)$ for varying $\lambda$. Indeed, let $F$ be an element of $\MI$ of weights $r, s$. Then, for each $s \in \mathbb C$, we have
$$L^*_F(s)=\sum_{k=0}^{\text{\rm min}(r, s)} L^*_{F_k}(s)$$
where $F_k$ are elements of $\HM$ of weight $r, s$ such that $F=F_0+\dots+F_{\text{\rm min}(r, s)}$ as in Prop. \ref{char}.

\subsection{L-functions in $\M_{r, s}$.} 
We now consider the case that $f$ is of polynomial growth at the cusps, i.e. $f \in \M_{r, s}$. Then, $\mathring{f}=0$ and $\tilde f=f-f^0.$ Further, for $\Re(w) \gg 0$, the integral 
$\int_0^{\infty}\tilde f(it)t^{w-1}dt$ converges and therefore we can make the change of variables $t \to 1/t$ in the third integral of \eqref{Lf} to derive
\begin{align*}L_f^*(w) &=
\int_1^{\infty} (f(it)-f^0(it)) t^w \frac{dt}{t}+i^{r-s}\int_0^{1} (f(i/t)-f^0(i/t)) t^{-r-s+w}\frac{dt}{t} \\
&=\int_0^{\infty} (f(it)-f^0(it)) t^{w}\frac{dt}{t} 
\end{align*}
Here we used the transformation law for $f$ and the formula for the antiderivative  of $f^0.$

This coincides with Brown's definition of L-functions of $f  \in \M_{r, s}$ given in \cite{BrI} (Sect. 9.4). There, up to a different normalisation, the L-function is actually defined, for Re$(w) \gg 0$,  by 
\begin{equation}\label{alterndef}
L_f^*(w) =\sum_{|j| \le M}
 (2 \pi)^{-j-w} \Gamma(w + j)L_f^{(j)}(w + j) 
\end{equation}
where, with the notation of \eqref{FE},
$$L_f^{(j)}(w):=\sum_{N \ge 1} \frac{1}{N^w}\left ( \sum_{m+n=N} a^{(j)}_{m, n} \right ).$$
The equivalence of this with our definition 
%in the case of $\M_{r, s}$ 
is established in the proof of Th. 9.7 of \cite{BrI}.

\subsubsection{L-functions in $\HM(\lambda) \cap \M_{r, s}$}
When, in addition, $f \in \M_{r, s}$ is an eigen-function of the Laplacian, then 
by computing Mellin transforms as usual, $L_f^*(w)$ obtains a more familiar form, which however, is valid for $\Re(w) \gg 0$. Specifically, let $f \in \HM(\lambda) \cap \M_{r, s}$. By Lemma \ref{HM!}, there is a $k_0 \in \ZZ$ such that $f=f^0+f^h+f^a$ with
$$ f^0(z)= ay^{k_0}+by^{1-r-s-k_0}, \quad f^h(z)=\sum_{j=k_0}^{-s} y^j\sum_{\substack{m >0}}a_{m}^{(j)} q^m \quad \text{and} \, \,
f^a(z)=\sum_{j=k_0}^{-r} y^j\sum_{\substack{m > 0}} b_{m}^{(j)} \bar q^m $$ 
($a, b, a_{m}^{(j)},  b_{m}^{(j)} \in \C $).
Then for $\Re(w) \gg 0$, \eqref{alterndef} (or, directly, \eqref{Lf}) becomes
%\begin{align*}
\[ L_f^*(w) =
%\int_1^{\infty} (f(it)-f^0(it)) t^w \frac{dt}{t}+i^{r-s}\int_1^{\infty} (f(it)-f^0(it)) t^{r+s-w}\frac{dt}{t}+P(w) \\
%&=\int_1^{\infty} (f(it)-f^0(it)) t^{w}\frac{dt}{t}-\int_1^0 (i^{r-s}f(\frac{i}{t})t^{-r-s}-f^0(it)) t^{w}\frac{dt}{t} \\&=
\sum_{j=-r-s}^{-s} \frac{\Gamma(j+w)}{(2 \pi )^{j+w}}\sum_{m>0}\frac{a_m^{(j)}}{m^{j+w}}+\sum_{j=-r-s}^{-r} \frac{\Gamma(j+w)}{(2 \pi )^{j+w}}\sum_{m>0}\frac{b_m^{(j)}}{m^{j+w}}
\]
%\end{align*}

\subsubsection{Example: L-function of the double Eisenstein series} We can use the above representation of $L^*_f(s)$ and \eqref{FEis} to compute explicitly the L-function of the double Eisenstein series. For $\Re(w) \gg 0,$ we have
\begin{multline*} L_{\mathcal E_{r, s}}^*(w)= \frac{(2 \pi)^{r+s-w} \pi}{\Gamma(r+1)\zeta(r+s+1)}\left ( \sum_{j=-r-s}^{-s}2^j \Gamma(j+w)  \alpha_{-j-\frac{r+s}{2}}^+\sum_{m>0}\frac{\sigma_{r+s+1}(m)}{m^{w+1}}+ \right. \\
\left. \sum_{j=-r-s}^{-r} 2^j 
\Gamma(j+w)  \alpha_{-j-\frac{r+s}{2}}^-\sum_{m>0}\frac{\sigma_{r+s+1}(m)}{m^{w+1}} \right )
= \\
\frac{\zeta(w+1)\zeta(w-r-s) (2 \pi)^{r+s-w} \pi}{\Gamma(r+1)\zeta(r+s+1)}\left ( \sum_{j=-r-s}^{-s} 2^j \Gamma(j+w) \alpha_{-j-\frac{r+s}{2}}^+ +\sum_{j=-r-s}^{-r} 2^j
\Gamma(j+w) \alpha_{-j-\frac{r+s}{2}}^- \right ) 
\end{multline*}
(The sum $\sum_{m>0}\frac{\sigma_{r+s+1}(m)}{m^{w+1}}$ has been computed as in the case of L-functions of the usual Eisenstein series.) The last expression also gives the meromorphic continuation to the entire $w$-plane.

\section{Maass-Selberg forms}\label{PerMaass}
In \cite{LZ}, the authors extend the classical theory of period polynomials of (holomorphic) cusp froms by assigning a period function to Maass cusp forms of weight $0$. M\"uhlenbruch \cite{MTh} later generalised that to Maass cusp forms of real weight. One of the ways to define the period function, in both \cite{LZ} and \cite{MTh}, is based on a differential form called Maass-Selberg form. We recall its definition and some of its properties.

Let $f, g$ be smooth functions defined in an open subset $U$ of $\HH \cup \bar \HH$. For $z=x+iy$, set
$$\{f, g\}^+:=f(z)g(z)\frac{dz}{y} \qquad \text{and} \, \, \,  \{f, g\}^-:=f(z)g(z)\frac{d\bar z}{y}.$$
Let $k \in 2\ZZ$. The {\it Maass-Selberg form} is then defined by 
\begin{equation}\label{MSel} \eta_k(f, g):=\{\partial_{k/2} f, g\}^+ -\{f, \bar \partial_{k/2}g \}^-
\end{equation}
(We normalise slightly differently from \cite{MTh} because we use Brown's version of the Maass operators instead of the
operators $E^+_{2k}=2\partial_k$ and $E^-_{2k}=2 \bar \partial_{-k}$ used in \cite{MTh}.)

The next lemma summarises the properties of Maass-Selberg form we will be needing. 
\begin{lem}\label{MSprop} (Lemma 39 of \cite{MTh}) For each $\g \in \G$, we have
 \newline
1. If $\eta_k(f, g) \circ \g$ denotes the pull-back of the differential form $\eta_k(f, g)$ by the map $z \to \g z$ ($z \in U$), then we have
$$\eta_k(f, g) \circ \gamma=\eta_k(f|_k \g, g|_{-k} \g).$$
\noindent
2. Suppose that, for some $\lambda \in \R$, we have $\Omega_k f=\lambda f$ and $\Omega_{-k} g=\lambda g$.
Then $\eta_k(f, g)$ is closed.
\end{lem}

\subsection{A Maass-Selberg form associated to modular iterated integrals of length one}
To define the Maass-Selberg form that we will associate to modular iterated integrals of length one we need
a function $R_{n, \nu}$ ($n \in 2\ZZ, \nu \in \C$)  we now define. For $z \in \HH \cup \bar \HH$ and $\zeta \ne z, \bar z$, it is given by 
$$R_{n, \nu}(z, \zeta)=\left ( \frac{\zeta-\bar z}{\zeta-z}\right )^{\frac{n}{2}} \left (\frac{ \text{Im} z}{(\zeta-z)(\zeta-\bar z)}\right )^{\frac12-\nu}.$$
For each $\zeta \in \C$, this gives a well-defined real-analytic function of $z$ if we restrict $z$ to the complement in $\HH$ of some path joining $\zeta$ and $\bar \zeta$ and then choose an appropriate branch for the implied logarithm. Likewise, for a suitable subset of $\bar \HH.$
%We denote this domain by $C_{\zeta}.$

For the specific values of $n, \nu$ we will use the function  $R_{n, \nu}$, it can be defined for all $\zeta \in \C$ and $z \in \HH$. Specifically, for $n=s-r$ and $\mu_k=-k+(r+s+1)/2$ with $k$ as in Prop \ref{char} we have, 
$$R_{n, \mu_k}(z, \zeta)=\left ( \frac{\zeta-\bar z}{\zeta-z}\right )^{\frac{s-r}{2}} \left (\frac{ \text{Im} z}{(\zeta-z)(\zeta-\bar z)}\right )^{k-\frac{r+s}{2}}=(\text{Im} z)^{k-\frac{r+s}{2}} (\zeta-z)^{r-k}(\zeta-\bar z)^{s-k}.$$
Since $k \le r, s$, this can be defined for all $z \in \HH \cup \bar \HH.$

\iffalse Note that in \cite{MTh}, the function $R_{n, \mu_k}$ is defined only for $\zeta \in \R$ for simplicity's sake, especially since $n/2$ is not necessarily integral making the argument computations complicated for general $\zeta$.
In our case, $n/2$ is integer and we can define $R_{n, \mu_k}$ in the same domain as \cite{LZ}. In fact, in our setting, $n$ and $\mu_k$ are such that $\zeta-z$ and $\zeta-\bar z$ eventually appear only as factors raised in a positive integer exponent, allowing us to define $R_{n, \mu_k}$ for all $z \in \HH.$
\fi

The function $R_{n, \mu_k}$ satisfies
\begin{lem}\label{Rprop} Set $n=s-r$ and $\mu_k=-k+(r+s+1)/2$ with $k \le r, s$.
 \newline
1.  For each $\zeta \in \C$ we have
\begin{align*} \bar \partial_{-\frac{n}{2}}R_{n, \mu_k}(\cdot, \zeta)&=\frac12 (1-2\mu_k-n)R_{n-2, \mu_k}(\cdot, \zeta)
 \\
\Omega_n R_{n, \mu_k}(\cdot, \zeta)&=\left ( \frac14-\mu_k^2 \right )R_{n, \mu_k}(\cdot, \zeta).
\end{align*}
2. For each $\g \in \G$ and $\zeta \in \HH$ we have 
$$R_{n, \mu_k}(\g z, \g \zeta)=j(\g, \zeta)^{1-2\mu_k} \left( \frac{j(\g, z)}{|j(\g, z|} \right )^n R_{n, \mu_k}(z, \zeta).$$
\end{lem}
\begin{proof}
This is essentially Prop. 36 of \cite{MTh} but there it is proved with the restriction that $\zeta \in \R$ and $j(\g, \zeta)>0$
due to the more general $\mu$ and $n$ to which the proposition applies.
\end{proof}

We now associate to $F_k$ Maass -Selberg forms which will be the basis for our construction of the period function for all functions in $\MI$.
\begin{prop}\label{MSmodint} Let $k \in \{0, \dots, \text{\rm min}(r, s) \}$ and $\mu_k=-k+(r+s+1)/2$. For each $\zeta \in \C$, the forms 
$\eta_{r-s}(y^{\frac{r+s}{2}} \left ( \tilde F_k+ay^{k-r-s} \right ), R_{s-r, \mu_k}(\cdot, \zeta) ), \quad 
\eta_{r-s}(y^{\frac{r+s}{2}} \mathring{F}_k, R_{s-r, \mu_k}(\cdot, \zeta) )$
and $\eta_{r-s}(y^{\frac{r+s}{2}} \left ( by^{1-k} \right ), R_{s-r, \mu_k}(\cdot, \zeta) )$
are closed. 
\iffalse
For $f=\tilde F_k$ or $\mathring{F}_k$ we have the explicit expression
\begin{multline}\label{explMS}
y^{k-1} \partial_r \tilde  F_k(z)(\zeta-\bar z)^{s-k}(\zeta-z)^{r-k}dz+(s-k)y^{k-1} \tilde F_k(z)(z))(\zeta-\bar z)^{s-k-1}(\zeta-z)^{r-k+1}d\bar z\\
+
\int_{i}^{\zeta}y^{k-1} \partial_r \mathring{F}_k (z) (\zeta-\bar z)^{s-k}(\zeta-z)^{r-k}dz+(s-k)y^{k-1}\mathring{F}_k(z)(\zeta-\bar z)^{s-k-1}(\zeta-z)^{r-k+1}d\bar z.
\end{multline}
\fi
\end{prop}
\begin{proof}
By Prop. \ref{Fk}, $y^{\frac{r+s}{2}} \left ( \tilde F_k+ay^{k-r-s} \right )$, $y^{\frac{r+s}{2}} \mathring{F}_k, y^{\frac{r+s}{2}}\left (  by^{1-k} \right ), $ are eigenfunctions of $\Omega_{r-s}$ and, 
by Lemma \ref{Rprop}, $R_{s-r, \mu_k}(\cdot, \zeta)$ is an eigenfunction of $\Omega_{s-r}$. They all have eigenvalue $\frac14-\mu_k^2$.
Therefore, with Lemma \ref{MSprop} we deduce the assertion.
% $\eta_{r-s}(y^{\frac{r+s}{2}}\tilde{F}_k, R_{s-r, \mu_k}(\cdot, \zeta) )$ (resp.
%\eta_{r-s}(y^{\frac{r+s}{2}}\mathring{F}_k, R_{s-r, \mu_k}(\cdot, \zeta) )$) is closed.
\end{proof}

\section{Cocycles associated to modular iterated integrals of length one.}
We briefly recall the basic cohomological formalism we will need. Let $M$ be a right $\Gamma$-module. If, for a non-negative integer $i$, $C^i(\Gamma, M)=\{s: \G^i \to M\}$ 
denotes the space of $i$-cochains for $\Gamma$  with coefficients in $M$, we define the differential $d^i\colon C^i(\Gamma, M) \to
C^{i+1}(\Gamma, M)$ by
\begin{align}\label{differential}
\begin{split}
&(d^i \sigma)(g_1, \dots, g_{i+1}):=\\
&
\sigma(g_2,\dots,g_{i+1}).g_1
+\sum_{j=1}^i (-1)^j \sigma(g_1,\dots, g_{j+1}g_j, \dots, g_{i+1})
+(-1)^{i+1}\sigma(g_1,\dots,g_{i}).
\end{split}
\end{align}
Then, we define $Z^i(\G, M):=\text{Ker}(d^i)$ (group of $i$-cocyles), $B^i(\G, M):=\text{Im}(d^{i-1})$ (group of $i$-coboundaries)
and $H^i(\G, M):=Z^i(\G, M)/B^i(\G, M)$ ($i$-th cohomology group of $M$).

In {\it Eichler cohomology}, the module $M$ is the space $P_m(K)$ of polynomial functions of degree $\le m$ and coefficients in a field $K$, acted upon by $\underset{-m, 0}{||}$.
An important theorem is the Eichler Shimura isomorphism
\begin{equation}\label{ESmap}
\phi: \overline {S_k} \oplus M_k \xrightarrow{\sim} H^1(\Gamma, P_{k-2}(\C))
\end{equation}
where $M_k$ (resp. $S_k$) is the space of classical holomorphic modular (resp. cusp) forms of weight $k$ for $\G.$
The isomorphism $\phi$ is induced by the assignment of $f \in M_k$ to the map $\phi(f): \G \to P_{k-2}(\C)$ such that
\begin{equation}\label{ES}
\phi(f)(\gamma)=\int_i^{\gamma^{-1} i} f(w)(w-z)^{k-2}dw \qquad \text{for $\gamma \in \Gamma$.}
\end{equation}

We will now associate to the $F_k$'s of the last section a $1$-cocycle in the $\G$-module $P_{r+s-2k}(\C)$. 
We define it as the coboundary of a $0$-cochain in a larger module than $P_{-2k+r+s}(\C)$. The construction
follows the definition of the ``integral at a tangential base point at
infinity" of \cite{Br}, (Section 4).

For convenience of notation, we set $\eta_{r-s}(g; \zeta):=\eta_{r-s}(y^{\frac{r+s}{2}} g, R_{s-r, \mu_k}(\cdot, \zeta) )$

\begin{prop} Let $k \in \{0, \dots, \text{\rm min}(r, s) \}$ and $\mu_k=-k+(r+s+1)/2$. The function $v_k: \HH \to \C$ given by
$$v_k(\zeta):=\int_{\zeta}^{i \infty}\eta_{r-s}(\tilde F_k+ay^{k-r-s}; \zeta )+
\int_{\zeta}^{0}\eta_{r-s}(by^{1-k}; \zeta) +
\int_{\zeta}^{-i \infty}\eta_{r-s}( \mathring{F}_k; \zeta),$$
where the line of integration in the last integral includes the origin, is well-defined. 
The differential forms to be integrated in $v_k$ can be written more explicitly in the form
\begin{equation}\label{expl}
y^{k-1} \partial_r f(z)(\zeta-\bar z)^{s-k}(\zeta-z)^{r-k}dz+(s-k)y^{k-1} f(z)(\zeta-\bar z)^{s-k-1}(\zeta-z)^{r-k+1}d\bar z
\end{equation}
 for each smooth function $f: \HH  \cup \bar \HH \to \C$. 
\end{prop}
\begin{proof} We first show the second assertion. With the definition of $\eta_{r-s}$ and Lem. \ref{Rprop} we have:
\begin{multline*}\eta_{r-s}(y^{\frac{r+s}{2}}f, R_{s-r, \mu_k}(\cdot, \zeta) )=
\left ( 2i y \frac{\partial}{\partial z}\left ( y^{\frac{r+s}{2}}f(z)\right )+\frac{r-s}{2}y^{\frac{r+s}{2}}f(z)\right ) R_{s-r, \mu_k}(\cdot, \zeta) \frac{dz}{y}\\
- y^{\frac{r+s}{2}}f(z)\left (\frac{1-2 \mu_k -s+r}{2}\right ) R_{s-r-2, \mu_k}(\cdot, \zeta) \frac{d\bar z}{y}=\\
\left (r y^{\frac{r+s}{2}}f(z)+2iy^{\frac{r+s}{2}+1}\frac{\partial f}{\partial z}(z)\right ) R_{s-r, \mu_k}(\cdot, \zeta) \frac{dz}{y}
- y^{\frac{r+s}{2}}f(z)\left (\frac{1-2 \mu_k -s+r}{2}\right ) R_{s-r-2, \mu_k}(\cdot, \zeta) \frac{d\bar z}{y}
\end{multline*}
Substituting the value for $\mu_k$ we get \eqref{expl}.
From this we deduce that, if $f$ and $\partial f/\partial z$ decay exponentially as $y \to \infty$, the same holds for $\eta_{r-s}(y^{\frac{r+s}{2}}f, R_{s-r, \mu_k}(\cdot, \zeta) )$. This condition holds for $f=\tilde F_k.$ It also holds for $f=\mathring{F_k}$ as $y \to -\infty.$ 

The term corresponding to $ay^{k-r-s}$ in the first integral is $O(y^{-2})$ as $y \to \infty$, which assures convergence. (Note that each of the two summands in \eqref{expl} individually has a term of order $y^{-1}$ but they cancel each other out on the upper imaginary axis). 

Since it is clear that the second integral in the definition of $v_f$ is convergent too, we deduce that, for each $\zeta \in \C$, all integrals are convergent.

Further, by Prop. \ref{MSmodint}, $\eta_{r-s}(\tilde F_k+ay^{k-r-s}; \zeta) $,   $\eta_{r-s}(by^{1-k}; \zeta)$ and
$\eta_{r-s}( \mathring{F}_k, \zeta )$ are closed in $\HH.$ The last form is also closed in $\bar \HH.$ Indeed, for each fixed $\zeta \in \C$, by \eqref{expl}, we have that $d(\eta_{r-s}( \mathring{F}_k, \zeta ))=P[e^{-2 \pi i z}]dz \wedge d \bar z$, where $P$ is a polynomial in $e^{-2 \pi i z}$ whose  coefficients are polynomials in $z, \bar z$. Since $\eta_{r-s}( \mathring{F}_k, \zeta )$
is closed in $\HH,$ each of those polynomials are identically zero in $\HH$ and therefore, they vanish in $\bar \HH$ too. 
%The last assertion of the proposition follows from \eqref{expl} applied to $f=\tilde F_k, F^0_k$ and $\mathring{F}_k.$
\end{proof}

We can now define the $1$-cocyle $\sigma_k$ on $\G$ as
$$\sigma_k=d^0 v_k.$$
We will show that, although $v_k$ does not belong to $P_{r+s-2k}(\C),$ its differential does and, in fact, it belongs to a cohomology class analogous to that of \eqref{ES} in the classical Eichler cohomology.

\begin{prop} \label{sigma}Let $F \in \MI \cap \M_{r, s}$. For $k \in \{0, \dots, \text{\rm min}(r, s)\},$ let $F_k$ be the $k$-th term in the decomposition of $F$ in eigenfunctions of $\Delta_{r, s}$ as in Prop. \ref{char}.  Then \newline
1. The map $\sigma_k$ induces a $1$-cocycle in $P_{r+s-2k}(\C)$. \newline
2. Let $\tilde \sigma_k: \G \to P_{r+s-2k}(\C)$ be the map given by 
$$\tilde \sigma_k(\g)(\zeta)=\int_i^{\g^{-1} i} \eta_{r-s}(y^{\frac{r+s}{2}}  F_k, R_{s-r, \mu_k}(\cdot, \zeta) ).$$
This gives a $1$-cocycle which belongs to the same cohomology class as $\sigma_k$.
\end{prop}
\begin{proof} We occasionally use again the abbreviation  $\eta_{r-s}(g; \zeta):=\eta_{r-s}(y^{\frac{r+s}{2}} g, R_{s-r, \mu_k}(\cdot, \zeta) )$. 

Since $\eta_{r-s}( g; \zeta)$ are closed for $g=\tilde F_k, ay^{k-r-s}, by^{1-k}$ and $\mathring{F}_k$, we have
\begin{multline}\label{cocycle0}
 v_k(\zeta)=\int_{\zeta}^i\eta_{r-s}(F_k-by^{1-k}-\mathring{F}_k; \zeta )+
\int_i^{i \infty}\eta_{r-s}( \tilde F_k+ay^{k-r-s}; \zeta)+\int_{\zeta}^0\eta_{r-s}(by^{1-k}; \zeta )\\
+\int_{\zeta}^{-i \infty}\eta_{r-s}( \mathring{F}_k; \zeta )=\\
\int_{\zeta}^i\eta_{r-s}(  F_k; \zeta )
+
\int_i^{i \infty}\eta_{r-s}( \tilde F_k+ay^{k-r-s}; \zeta )+\int_i^0\eta_{r-s}(by^{1-k}; \zeta)
+\int_i^{-i \infty}\eta_{r-s}( \mathring{F}_k; \zeta )
\end{multline}
where the last integral is also taken to be over a path that includes the origin.
By \eqref{expl}, the last three terms of \eqref{cocycle0} are clearly in $P_{r+s-2k}(\C)$. (However, note that to reach this conclusion, we first fix a specific path of integration and only then expand the integrand in $\zeta$. If we first expanded, there would be no guarantee {\it a priori} that the resulting differentials are closed). Thus, the image of those integrals under the action by $\underset{2k-r-s, 0}{||}(\gamma-1)$ is in $P_{r+s-2k}(\C)$ too.
  
To show that $\int_{\zeta}^i\eta_{r-s}(  F_k; \zeta )\underset{2k-r-s, 0}{||}(\gamma-1) \in P_{r+s-2k}(\C),$ we observe that, by Lemma \ref{MSprop} (1), we have, for each $\g \in \G$,
$$
\eta_{r-s}(y^{\frac{r+s}{2}}  F_k, R_{s-r, \mu_k}(\cdot, \g \zeta) ) \circ \g=
\eta_{r-s}((y^{\frac{r+s}{2}}  F_k)|_{r-s}\g, R_{s-r, \mu_k}(\cdot, \g \zeta)|_{s-r}\g )$$
where the action $|_{s-r}$ refers to the implied variable $z$. By Prop. \ref{Fk} and Lemma \ref{Rprop}(2), the last term
equals $\eta_{r-s}(y^{\frac{r+s}{2}}  F_k, R_{s-r, \mu_k}(\cdot, \zeta) )j(\g, \zeta)^{1-2\mu_k}$.
Therefore
\begin{align*}
&\left ( \int_{\zeta}^i\eta_{r-s}(y^{\frac{r+s}{2}}  F_k, R_{s-r, \mu_k}(\cdot, \zeta) )\right )\underset{1-2\mu_k, 0}{\Big | \Big |}\g= \nonumber
\left ( \int_{\g \zeta}^i\eta_{r-s}(y^{\frac{r+s}{2}}  F_k, R_{s-r, \mu_k}(\cdot, \g \zeta) )\right )j(\g, \zeta)^{2\mu_k-1}
\\=
&\left ( \int_{\zeta}^{\g^{-1} i}\eta_{r-s}(y^{\frac{r+s}{2}}  F_k, R_{s-r, \mu_k}(\cdot, \g \zeta)  ) \circ \g \right )j(\g, \zeta)^{2 \mu_k-1}= \int_{\zeta}^{\g^{-1} i} \eta_{r-s}(y^{\frac{r+s}{2}}  F_k, R_{s-r, \mu_k}(\cdot, \zeta) )
\end{align*}
This implies that 
\begin{equation}\label{cocycle}
\left ( \int_{\zeta}^i\eta_{r-s}(y^{\frac{r+s}{2}}  F_k, R_{s-r, \mu_k}(\cdot, \zeta) )\right )\underset{1-2\mu_k, 0}{\Big | \Big |}(\g-1)= \int_i^{\g^{-1} i} \eta_{r-s}(y^{\frac{r+s}{2}}  F_k, R_{s-r, \mu_k}(\cdot, \zeta) )
\end{equation}
which is in $P_{r+s-2k}(\C).$ 

Therefore $\sigma_k(\gamma) \in P_{r+s-2k}(\C)$ and since, further, $\sigma_k$ is given as the differential of a $0$-cochain, we deduce $1$.

To derive $2.$ we see with \eqref{cocycle0} and \eqref{cocycle} that, for all $\g \in \G$ and $\zeta \in \HH,$
\begin{multline}\label{cohequiv}
\sigma_k(\gamma)(\zeta)=\tilde \sigma_k(\g)(\zeta)+\\
\left ( 
\int_i^{i \infty}\eta_{r-s}( \tilde F_k+ay^{k-r-s}; \zeta )+
\int_i^{-i \infty}\eta_{r-s}(\mathring{F}_k; \zeta )+
\int_i^{0}\eta_{r-s}(by^{1-k}; \zeta )
\right ) \underset{2k-r-s, 0}{\Big | \Big |}(\g-1).
\end{multline}
Since the last integrals belong to $P_{r+s-2k}(\C),$ we deduce, on the one hand, that $\tilde \sigma_k$ is a $1$-cocycle with coefficients in $P_{r+s-2k}(\C)$ and, on the other, that $\sigma_k$ and $\tilde \sigma_k$ differ by a $1$-coboundary in $P_{r+s-2k}(\C)$. Therefore the belong to the same cohomology class.
\end{proof}

As in the case of the classical period polynomial, the value of this cocycle at the involution 
$S$
encapsulates the critical values of the L-functions of $F_k$. However, in the general case, its leading and constant terms must be ``truncated".

\begin{thm} \label{Lperiod}Assume that $r \equiv s \mod 4.$ Then,
$$\sigma_k(S)(\zeta)-\sigma_k^0(S)(\zeta)=i \sum_{l=1}^{s+r-2k-1} \left ( \sum_{n=0}^l \alpha_n \right ) L^*_{F_k}(k+l) \zeta^l$$
where $\sigma_k^0(S)(\zeta)$ denotes the sum of the leading and constant term of $\sigma_k(\zeta)$ and
$$\alpha_n=i^{-l-2n} \left ( (r-k) \binom{s-k+1}{l-n} \binom{r-k+1}{n}-(s-k) \binom{s-k-1}{l-n} \binom{r-k+1}{n} \right )$$
\end{thm}
\begin{proof}
We notice, with \eqref{cohequiv} that $\sigma_k(S)(\zeta)$ equals
\begin{equation}\label{cohequiv1}
\left ( 
\int_i^{i \infty}\eta_{r-s}(\tilde F_k+ay^{k-r-s}; \zeta) +
\int_i^{-i \infty}\eta_{r-s}(\mathring{F}_k; \zeta)+
\int_i^0\eta_{r-s}(by^{1-k}; \zeta)
)\right ) \underset{2k-r-s, 0}{\Big | \Big |}(S-1)
\end{equation}
Now, each smooth function $h: \HH  \cup \bar \HH \to \C$, we deduce from \eqref{expl} that 
\begin{multline}\label{intparts}
\eta_{r-s}(h; \zeta)=2i\frac{\partial}{\partial z} \left ( y^k h(z)(\zeta-\bar z)^{s-k} (\zeta-z)^{r-k} \right)+(r-k)h(z)
(\zeta-\bar z)^{s-k+1} (\zeta-z)^{r-k-1}\\
-(s-k)h(z)(\zeta-\bar z)^{s-k-1} (\zeta-z)^{r-k+1}
 \end{multline}
We further notice that $\eta_{r-s}( y^{k-r-s}; 0)=0$.
Therefore,  $\eta_{r-s}( y^{k-r-s}; \zeta)=\eta_{r-s}( y^{k-r-s}; \zeta)- \eta_{r-s}( y^{k-r-s}; 0)$  equals
\begin{multline}\label{intparts1}
2i\frac{\partial}{\partial z} \left ( y^k \cdot  y^{k-r-s} \left ( (\zeta-\bar z)^{s-k} (\zeta-z)^{r-k}- \bar z^{s-k} z^{r-k}\right ) \right)\\
+(r-k)\cdot  y^{k-r-s}
\left ( (\zeta-\bar z)^{s-k+1} (\zeta-z)^{r-k-1}- \bar z^{s-k+1} z^{r-k-1}\right )\\
-(s-k)\cdot  y^{k-r-s}\left ( (\zeta-\bar z)^{s-k-1} (\zeta-z)^{r-k+1}- \bar z^{s-k-1} z^{r-k+1}\right ).
 \end{multline}
Each of the polynomials in $\zeta$ have degree $\le s+r-2k-1$. Hence the integral 
$\int_i^{i \infty}\eta_{r-s}(y^{k-r-s}; \zeta)$ converges and, by integrating along the positive imaginary axis, we deduce that it equals
\begin{multline}\label{yterm} i \int_1^{\infty} y^{2k-r-s-1}  \Big [ (r-k) \left ( (\zeta+ iy)^{s-k+1}(\zeta-iy)^{r-k-1}- (iy)^{s-k+1}(-iy)^{r-k-1} \right) -\\
 (s-k) \left ( (\zeta+ iy)^{s-k-1}(\zeta-iy)^{r-k+1} - (iy)^{s-k+1}(-iy)^{r-k-1} \right ) \Big ] dy
-2i ( (\zeta+i)^{s-k}(\zeta-i)^{r-k} -1)
\end{multline}
Equ. \eqref{intparts} can be used directly for $h=\tilde F_k, \mathring{F}_k, by^{1-k}$, to yield
\begin{multline*}\label{}
 \int_i^{i \infty}\eta_{r-s}(y^{\frac{r+s}{2}} h, R_{s-r, \mu_k}(\cdot, \zeta) )=(r-k)i
 \int_1^{\infty} t^{k-1} h(it) (\zeta+ it)^{s-k+1}(\zeta-it)^{r-k-1}dt\\
-(s-k)i  \int_1^{\infty} t^{k-1} h(it) (\zeta+ it)^{s-k-1}(\zeta-it)^{r-k+1}dt-2ih(i)(\zeta+i)^{s-k}(\zeta-i)^{r-k}.
\end{multline*}
 Therefore, with \eqref{cohequiv1} and \eqref{yterm} we deduce that  $\sigma_k(S)(\zeta)-\sigma^0_k(S)(\zeta)$ equals
\begin{multline}\label{expl2}
i \int_1^{\infty} t^{k-1} \left ( \tilde F_k(it)+at^{k-r-s} \right )  (R(t, \zeta)-R_0(t, \zeta)) dt+ 
i \int_1^{-\infty} t^{k-1} \mathring{F}_k(it) (R(t, \zeta) -R_0(t, \zeta))dt \\
+
i \int_1^0 t^{k-1} \left ( bt^{1-k} \right )(R(t, \zeta) -R_0(t, \zeta))dt
\end{multline}
where
\begin{multline*} R(t, \zeta):=  (r-k) \left ( 
(1-it \zeta)^{s-k+1}(1+i t \zeta)^{r-k-1}-(\zeta+it)^{s-k+1}( \zeta-it)^{r-k-1} \right ) 
\\
-(s-k)
\left (  (1-it \zeta)^{s-k-1}(1+i t \zeta)^{r-k+1}-(\zeta+it)^{s-k-1}(\zeta-it)^{r-k+1} \right ) 
%+2i\tilde F_k(i)(\zeta+i)^{s-k}(\zeta-i)^{r-k}
\end{multline*}
and $R_0(t, \zeta)$ is the sum of the constant and the leading term of the expansion of $R(t, \zeta)$ in $\zeta$.
(The terms  of order between $1$ and $s+r-2k-1$ coming from $2ih(i)(\zeta+i)^{s-k}(\zeta-i)^{r-k}$ and $2i ( (\zeta+i)^{s-k}(\zeta-i)^{r-k} -1)$ cancel because $s \equiv r \mod 4.$) 

Using the binomial expansion, we see that the integral in the RHS of \eqref{expl2} equals, in the notation of the statement of the proposition:

\begin{multline*} \sum_{l=1}^{s+r-2k-1} \left ( \sum_{n=0}^l \alpha_n  \right ) \left ( \int_1^{\infty} \left ( \tilde F_k(it)+at^{k-r-s} \right ) \left ( t^{k+l}+i^{s-r}t^{s+r-k-l}\right ) \frac{dt}{t} + \right. \\ 
\left. \int_1^{-\infty}  \mathring{F}_k(it) \left ( t^{k+l}+i^{s-r}t^{s+r-k-l}\right ) \frac{dt}{t}  +
 \int_1^{0} bt^{1-k}  \left ( t^{k+l}+i^{s-r}t^{s+r-k-l}\right ) \frac{dt}{t}  
\right ) \zeta^l.
\end{multline*}  With the definition of $L_{F_k}^*$ we deduce the proposition.
\end{proof}
 
From Prop.\ref{sigma} we obtain a map from the space of modular iterated integrals of length one to a direct sum of copies of the space of classical modular (resp. cusp) forms. 

\begin{thm} The maps $\sigma_k$ ($k \in \{0, \dots, \text{\rm min}(r, s)\}$) defined in Prop. \ref{sigma} induce a map
$$\MI \cap \M^!_{r, s}\to \bigoplus_{k=0}^{\text{\rm min}(r, s)} H^1(\G, P_{r+s-2k}(\C)) \cong  \bigoplus_{k=0}^{\text{\rm min}(r, s)} 
\left ( \bar S_{r+s-2k+2} \oplus M_{r+s-2k+2} \right ).$$
\end{thm}
\begin{proof} Prop. \ref{sigma} induces a map sending each $F \in \MI \cap \M^!_{r, s}$ to 
$$([\sigma_0], \dots [\sigma_{\text{\rm min}(r, s)}]) \in \bigoplus_{k=0}^{\text{\rm min}(r,s)} H^1(\G, P_{r+s-2k}(\C)).$$
Here $[\sigma_k]$ stands for the cohomology class of the $1$-cocycle defined in that proposition. The last isomorphism of the theorem follows from the Eichler-Shimura isomorphism (see \eqref{ESmap}).
\end{proof}

%{\bf Remark.} A natural question is to determine the kernel of the map of the last theorem. 

\begin{cor} Let $F $ be a modular iterated integral of length one and weights $(r, s)$ and let $F=F_0+\dots+F_{\text{\rm min}(r, s)}$ be its decomposition into eigenfunctions of the Laplacian. Then, for each $k \in \{0, \dots, \text{\rm min}(r, s)\}$, there is a $P_k(\zeta) \in P_{r+s-2k}(\C)$ and unique $f_k \in S_{r+s-2k+2}, g_k \in M_{r+s-2k+2}$ such that, for all $\g \in \G,$
\begin{multline}\int_i^{\g^{-1} i} \eta_{r-s}(y^{\frac{r+s}{2}}  F_k, R_{s-r, \mu_k}(\cdot, \zeta) )=
\int_i^{\g^{-1} i} f(z)(z- \zeta)^{r+s-2k}dz+
\overline{\int_i^{\g^{-1} i} g(z)(z-\bar \zeta)^{r+s-2k}dz}\\
+P_k\underset{2k-r-s, 0}{\Big | \Big |}(\g-1).
\end{multline}
\end{cor}

\section{An application to algebraicity}\label{alquest}
In \cite{BH} an Eichler-Shimura isomorphism for weakly holomorphic modular forms is proved, which respects rational structures. As a proof of concept for the ``correctness" of our definition of the L-function in Sect. \ref{Lfunctions} we will use the results of \cite{BH} to show an analogue of Manin's Periods Theorem \cite{M} for weakly holomorphic forms. It should be mentioned that K. Bringmann has shown us an alternative way, based on results of \cite{BGKO}, to establish a statement implying the same result. 

Before stating and proving our result, we first summarize the setup of \cite{BH} and then show that it is compatible with the explicit expressions for the cocycles of the last section. 

Let $M_{k, \mathbb Q}^!$, resp. $S_{k, \Q}^!$, denote the $\Q$-vector space of weight $k$ weakly holomorphic modular, resp. cusp, forms for $\Gamma=$SL$_2(\ZZ),$ with rational Fourier coefficients. Consider the differential operator $D=\frac{1}{2 \pi i}\frac{d}{dz}$. In \cite{G} it is shown that, although there are generally no Hecke eigenforms in  $S_{k, \Q}^!$, there are well-defined operators on $M_{k, \Q}^!/D^{k-1}M_{2-k, \Q}^!$ induced by the standard Hecke operators and, within that space, there are Hecke invariant classes. 
With this terminology and notation, we have 
\begin{thm} \label{BHES} (Cor. 1.3 of \cite{BH}) The map $\phi$ assigning to $f \in M_k^!$ the function \eqref{ES} induces a Hecke invariant isomorphism 
$$[\phi]: M_{k, \Q}^!/D^{k-1}M_{2-k, \Q}^! \otimes_{\mathbb Q} \C \xrightarrow{\sim}  H^1(\Gamma, P_{k-2}(\mathbb Q)) \otimes_{\mathbb Q} \C$$
\end{thm}
The image of $S_{k, \Q}^!/D^{k-1}M_{2-k, \Q}^! \otimes_{\mathbb Q} \C$ is the \emph{parabolic} cohomology group defined by:
$$H^1_{\text{par}}(\Gamma, P_{k-2}(\mathbb C)):=Z^1_{\text{par}}(\Gamma, P_{k-2}(\mathbb C))/B^1_{\text{par}}(\Gamma, P_{k-2}(\mathbb C))$$
where $Z^1_{\text{par}}(\Gamma, P_{k-2}(\mathbb C)):=\{\psi \in Z^1(\Gamma, P_{k-2}(\mathbb C)); \psi(T)=0\}$ is the space of parabolic cocycles and $B^1_{\text{par}}(\Gamma, P_{k-2}(\mathbb C)):=\{\psi \in B^1(\Gamma, P_{k-2}(\mathbb C)); \psi(T)=0\}$ the space of parabolic coboundaries.
These spaces are defined over $\Q$ and $B^1_{\text{par}}(\Gamma, P_{k-2}(\mathbb C))$ is generated by $\psi_0$ such that
\begin{equation}\label{cbou}
\psi_0(\g)=1\underset{2-k, 0}{||}(\g-1).
\end{equation}

Further, let $F_{\infty}$ be the ``real Frobenius" induced by the map sending $\sigma \in Z^1_{\text{par}}(\Gamma, P_{k-2}(\mathbb C))$ to $F_{\infty} \sigma \in Z^1_{\text{par}}(\Gamma, P_{k-2}(\mathbb C))$ such that
$$F_{\infty}\sigma \left ( \begin{smallmatrix} a & b \\ c & d \end{smallmatrix} \right )(z)=\sigma \left ( \begin{smallmatrix} a & -b \\ -c & d \end{smallmatrix} \right )(-z).$$
Then $H^1_{\text{par}}(\Gamma, P_{k-2}(\mathbb C))$ is decomposed canonically into $F_{\infty}$-eigenspaces:
$$H^1_{\text{par}}(\Gamma, P_{k-2}(\mathbb C))=H^{1, +}_{\text{par}}(\Gamma, P_{k-2}(\mathbb C)) \oplus H^{1, -}_{\text{par}}(\Gamma, P_{k-2}(\mathbb C)).$$
Each class of $H^{1, +}_{\text{par}}(\Gamma, P_{k-2}(\mathbb C)) $ (resp. $H^{1, -}_{\text{par}}(\Gamma, P_{k-2}(\mathbb C)) $ is represented by a cocycle $\sigma$ such that $\sigma(S)$ is an even (resp. odd) polynomial.

Let now $\phi'$ be the map assigning to each $f \in S_{k, \Q}^! $ the cocyle $\sigma'_f=d^0v'_f$
where $v'_f:\frak H \to P_{k-2}(\C)$ is given by
$$v'_f(z):=\int_{i \infty}^{z}\tilde f(w)(w-z)^{k-2}dw+
\int_{-i \infty}^{z} \mathring{f}(w)(w-z)^{k-2}dw,$$
$\tilde f, \mathring{f}$ defined by \eqref{tilde}, \eqref{ring}, respectively. 

We will prove the proposition
\begin{prop} The map $\phi'$ induces a map $[\phi']$
% (denoted by the same notation) 
from the space $S_k^!/D^{k-1}M_{2-k}^! \otimes_{\mathbb Q} \C$ to
$ H^1_{par}(\Gamma, P_{k-2}(\mathbb Q)) \otimes_{\mathbb Q} \C$. The resulting diagram
\begin{equation} \label{diag}
\begin{tikzcd}
M_k^!/D^{k-1}M_{2-k}^! \otimes_{\mathbb Q} \C \arrow{r}{[\phi]}  & H^1(\Gamma, P_{k-2}(\mathbb Q)) \otimes_{\mathbb Q} \C \\
S_k^!/D^{k-1}M_{2-k}^! \otimes_{\mathbb Q} \C \arrow[u, hook, "i"]
\arrow{r}{[\phi']} &   H^1_{par}(\Gamma, P_{k-2}(\mathbb Q)) \otimes_{\mathbb Q} \C  \arrow[u, hook, "j"]
\end{tikzcd}
\end{equation}
(where $i, j$ are natural injections and $\phi$ is as in Th. \ref{BHES}) is commutative.  
\end{prop}
\begin{proof} Let $f \in S_{k, \Q}^!.$ Then since both $\tilde f$ and $\mathring{f}$ are periodic with period $1$, it is easy to see that $\sigma'_f(T)=v'_f\underset{2-k, 0}{||}(T-1)(z)=v'_f(z+1)-v'_f(z)=0.$ In addition, $\sigma'_f$ is a cocycle by construction and therefore it is a parabolic cocycle. This proves the first assertion.
 
Now,
$$v'_f(z)  =
\int_i^zf(w)(w-z)^{k-2}dw+\int_{i \infty}^i \tilde f(w)(z-w)^{k-2} dw+ \int_{-i \infty}^i \mathring{f}(w)(z-w)^{k-2} dw$$
and thus
\begin{multline} \label{rE}  \sigma'_f(\g)
= \int_{\g^{-1}i}^if(w)( \cdot-w )^{k-2}dw+ 
\\
\left [ \int_{i \infty}^i \tilde f(w)(\cdot-w)^{k-2} dw+ \int_{-i \infty}^i \mathring{f}(w)(\cdot-w)^{k-2} dw \right ] \underset{2-k, 0}{||}(\g-1).
\end{multline}
%and (for $\g=S$)
%\begin{equation}\label{rformula}
%r(f)=\left ( \int_{i \infty}^i \tilde f(w)(\cdot-w)^{k-2} dw+ \int_{-i \infty}^i \mathring{f}(w)(\cdot-w)^{k-2} dw \right ) \underset{2-k, 0}{||}(S-1). 
%\end{equation}
Since the term inside the square brackets belongs to $P_{k-2}(\C)$, the
cohomology class $j([\phi'](f))$ of this cocycle coincides with $[\phi](f)$.
\end{proof}
In \cite{BGKO}, (Th. 1.2.) it is proved that $[\phi']$ is an isomorphism. 

We are now ready to prove
\begin{thm}\label{ManinWp} Suppose that the class of $f \in S_k^!$ in $S_k^!/D^{k-1}M^!_{2-k}$ is an eigenclass of the Hecke operators. Let $K_f$ denote the field generated by the Fourier coefficients of $f$.Then there are $\omega^{\pm}(f) \in \C$ such that
$$L^*_f(j) \in \omega^+(f)K_f, \, \, \,  \text{for odd $j \in \{2, k-2\}$ and} \, \,  L^*_f(j) \in \omega^-(f)K_f, \, \, \,  \text{for even $j \in \{2, k-2\}$}.$$
\end{thm}
\begin{proof}
In \cite{G}, it is proved that the eigenspace of the class of $f$ in $S_{k, \Q}^!/D^{k-1}M_{2-k, \Q}^! \otimes_{\Q} \C$ is two-dimensional. It is defined over $K_f$. We let $V_f^{\text{deR}} \subset S_{k, \Q}^!/D^{k-1}M_{2-k, \Q}^! \otimes_{\Q} K_f $ denote the Hecke eigenspace generated over $K_f$ by $f$ and let $V_f^B$ be the corresponding eigenspace in $H^1(\Gamma, P_{k-2}(\Q)) \otimes_{\Q} K_f$. (We follow the notation of \cite{BH} to indicate the de Rham- (resp. Betti)-cohomological origin of those eigenspaces.) The space $V_f^B$ is two dimensional and defined over $K_f$ and therefore, so is the corresponding eigenspace $V^W_f$ in $H^1_{\text{par}}(\Gamma, P_{k-2}(\Q)) \otimes_{\Q} K_f$. 
It decomposes into invariant and anti-invariant eigenspaces with respect to the real Frobenius:
$V^W_f=V^{W, +}_f \oplus V^{W, -}_f.$
Further, by Th. \ref{BHES}, the map 
$$[\phi']: V_f^{\text{deR}} \otimes_{K_f} \C \xrightarrow{\sim} V_f^W \otimes_{K_f} \C$$
is a canonical isomorphism. 
Therefore, for some $\omega^{\pm}(f) \in \C,$ and some $\phi^{\pm}(f) \in V_f^{W, \pm} \subset H^{1, \pm}_{\text{par}}(\Gamma, P_{k-2}(K_f))$, we have,
$$[\phi'(f)]=\omega^{+}(f)\phi^+(f)+\omega^-(f) \phi^-(f).$$
Thus, for an even $P^+ \in Z_{\text{par}}^1(\Gamma, P_{k-2}(K_f))$,  an odd $P^- \in Z_{\text{par}}^1(\Gamma, P_{k-2}(K_f))$ 
and a $c_f \in \C,$ 
$$\phi'(f)=\omega^+(f)P^++\omega^-(f)P^-+c_f \psi_0.$$
(Recall that $\psi_0$ is defined by \eqref{cbou}). This gives
\begin{equation}\label{finform}
\phi'(f)(S)=\omega^+(f)P^+(S)+\omega^-(f)P^-(S)+c_f \underset{2-k, 0}{||}(S-1).
\end{equation} 
On the other hand, it is easy to see from  our definition of L-function and an application of the binomial formula to \eqref{rE}, that the coefficient of $z^{j-1}$ in $\phi'(f)(S)$ is a multiple of $L^*_f(j)$ by an element of $\Q[i]$. By comparing,  in \eqref{finform}, the coefficients of $z^{j-1}$, for $j$ odd (resp. even) other than $j=1,$ and $j=k-1$, we deduce the assertion.
\end{proof}

\end{document}